\documentclass[12pt]{article}
\usepackage{graphicx} 
\usepackage[toc,page]{appendix}
\usepackage{epstopdf}
\usepackage{subfigure}
\usepackage{color}
\usepackage[english]{babel}
\usepackage{amsmath,amsthm}
\usepackage{amsfonts}

\usepackage{graphicx}
\newtheorem{thm}{Theorem}[section]

\newtheorem{lem}[thm]{Lemma}

\newtheorem{conj}[thm]{Conjecture}

\title{Embedding the Complete Expansion Graph in Books}

\author{ Zeling Shao, Chunjin Ren, Zhiguo Li$^{*}$ \\
{\small School of Science, Hebei University of Technology, Tianjin 300401, China}
\date{}
\footnote{Corresponding author. E-mail: zhiguolee@hebut.edu.cn}
\footnote{This work is supported by the Natural Science Foundation of Hebei Province (No.A2019402043). }
}

\begin{document}
\baselineskip 0.65cm

\maketitle

\begin{abstract}

 A book embedding of a graph consists of an embedding of its vertices along the spine of a book, and an embedding of its edges on the pages such that edges embedded on the same page do not intersect.
 The pagenumber is the minimum number of pages in which the graph $G$ can be embedded. The main purpose of this paper is to study  the book embedding of the complete expansion graph. This is the first work about it, 
 and some exact pagenumbers of the complete expansion graphs of some special graphs are obtained.


\noindent\textbf{2000 MR Subject Classification.} 05C10
\end{abstract}

\section{Introduction}
~~~~~The problem of embedding graphs in books is an important research content of graph theory. 
P. C. Kainen$^{[1]}$ first introduced the book embedding and later investigated by F. Bernhart and P. C. Kainen$^{[2]}$. A book embedding of a graph $G$ consists of placing the vertices of G on a spine and assigning edges of the graph to pages so that edges in the same page do not cross each other. The order of the vertices is called the printing cycle of the embedding. The minimum number of pages in which a graph can be embedded is called the pagenumber of the graph$^{[2,3,4]}$, and the correspondent order of the vertices on the spine is called the optimal order of the minimal embedding.

Note that one can neglect the exact geometry, as two edges that are drawn on the same page cross if and only if their endpoints alternate along the spine. We say that an edge $e$ nests a vertex $v$ iff one endpoint of $e$ is to the left of $v$ along the spine and the other endpoint of $e$ to the right. We also say that an edge $e$ nests an edge $e^{'}$ iff both $e$ and $e^{'}$ are drawn on the same page and both endpoints of $e^{'}$ are nested by $e$. Nested edges do not cross.

Book embedding has been applied to VLSI design, sorting with parallel stacks, single-row routing,  fault-tolerant processor arrays, Turing  machine graphs and many other fields $^{[5]}$. Determining the pagenumber of an arbitrary graph embedded in book is NP-complete even if the order of vertices on the spine is fixed. Some pagenubmers of particular families of graphs have been discussed, for example: complete graphs$^{[2]}$, complete bipartite graphs$^{[6,7]}$, regular graphs$^{[8]}$, toroidal graphs$^{[9]}$, incomplete hypercubes$^{[10]}$, strong product of some graphs $^{[11]}$, and Cartesian product of some graphs$^{[12]}$, etc.

The concept of expansion graph was given in 2009 by A. Yongga and Siqin at first$^{[13]}$. It has caught the attention of some scholars to start studying the various properties of the expansion graph. Some scholars have studied the coloring  problems of the expansion graph$^{[14]}$. The hamiltonicity of the complete expansion graph has been studied in [15], etc.

In this paper, we mainly investigate the pagenumber of the complete expansion graph $\mathcal{E}_{c}(G)$ of $G$. Section 2 mainly introduces the definition and some properties of the complete expansion graph. In section 3, we get the relationship between the pagenumber of the complete expansion graph of the subgraph and the pagenumber of the complete expansion graph of the supergraph. For any simple graph $G$, the pagenumber of $\mathcal{E}_{c}(G)$ is related to the maximum degree of $G$ and pagenumber of $G$. We obtain the upper bound and the lower bound of the pagenumber of $\mathcal{E}_{c}(G)$. For the complete expansion graphs of some special graphs, such as star graph $S_{m}$, tree $T$, $M\ddot{o}bius$ Ladder $M_{h}$, Petersen graph $P$ and complete graph $K_{2m}$, the exact pagenumbers of them are obtained.



\section{Preliminaries}
~~~~~The graphs considered in this article are simple, connected and undirected.
Given a graph $G$, with edge set $E(G)$ and vertex set $V(G)$, we call that two vertices $i$ and $j$ are adjacent, if there is an edge $e=(i,j)\in E(G)$. The ends $i$ and $j$ are said to be incident with the edge $e$, and vice versa.
 The degree of a vertex $v$ in a graph $G$, denoted by $d_{G}(v)$, is the number of edges of $G$ incident with $v$.  Let $\Delta(G)$ be the maximum degrees of the vertices of $G$.  We denote the complete graph with $n$ vertices by $K_{n}$. It is a standard exercise to show that the pagenumber $pn(K_{n})$ of the complete graph $K_{n}$ is $\lceil\frac{n}{2}\rceil$ $^{[2]}$. 

If $G$ has vertices $v_{1},v_{2},\cdots,v_{n}$, the sequence $(d_{1}, d_{2}, ..., d_{n})$ is called a degree sequence of $G$. Let the set $A=V(G)\cup E(G)$, $B=\{K_{d_{1}},K_{d_{2}},\cdots,K_{d_{n}}\}\cup E(G)$, where $K_{d_{i}}$ represents the complete graph with $d_{i}$ vertices, $i=1,\cdots,n$. For a given graph $G$, the bijection $f$ of $A$ to $B$ is defined as:

\begin{equation} f(x)=
\begin{cases}
K_{d_{i}}~~~~~~~~x\in V(G),d_{G}(x)=d_{i};\\
x~~~~~~~~~~~~~~~~x\in E(G).
\end{cases}
\end{equation}

 $f$ is called the complete expansion transformation of $G$. The generated graph after complete expansion transformation is called the complete expansion graph of $G$, denoted by $\mathcal{E}_{c}(G)$. $\mathcal{E}_{c}(G)$ is a graph with vertex set $V(\mathcal{E}_{c}(G))=\{(v,m)|v\in V(G),m\in E(G)$, $v$ is incident with $m\}$, and edge set $E(\mathcal{E}_{c}(G))=E_{1}\cup E_{2}$, where $E_{1}=\{((v,m)(v,n))|v\in V(G),m,n\in E(G)\}, E_{2}=\{((v,m)(w,m))|v,w\in V(G),m\in E(G)\}$(see an example in Figure 1).

By the definition of complete expansion graph, it is easy to get the following properties:

\begin{lem}$^{[13]}$
For $x,y\in V(G)$, the complete expansion transformation $f$ of $G$ satisfies:

(1) $x\neq y\Leftrightarrow V(f(x))\cap V(f(y))=\emptyset$;

(2) $xy\in E(G)$, there is a unique $(x^{'}, y^{'})\in V(f(x))\times V(f(y)), x^{'} y^{'} \in E(\mathcal{E}_{c}(G))$ ;

(3) $\forall x\in V(G)$, $[V(f(x)),V(\mathcal{E}_{c}(G))\backslash V(f(x))]=d_{G}(x)$.


(4) 
For $xy, xz\in E(G), y\neq z$ ,
if $(x^{'}, y^{'})\in V(f(x))\times V(f(y))$ with $x^{'}y^{'}\in E(\mathcal{E}_{c}(G))$, then for any $z^{'}\in V(f(z))$, $x^{'}z^{'}\not\in E(\mathcal{E}_{c}(G))$.
\end{lem}

\begin{figure}[htbp]
\centering
\includegraphics[height=5cm, width=0.6\textwidth]{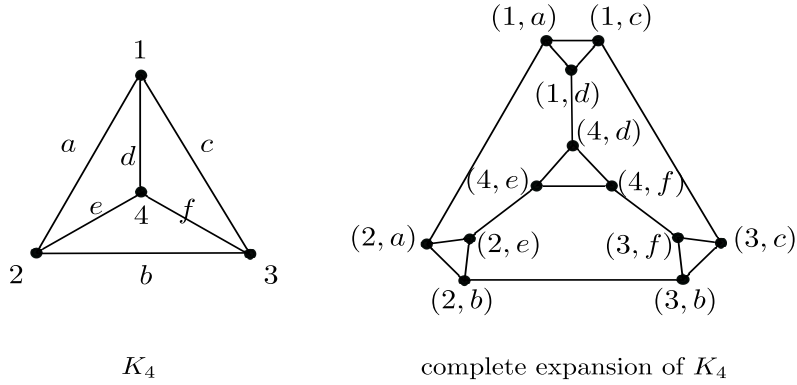}
\caption{$K_{4}$ and the complete expansion graph of $K_{4}$}
\label{1}
\end{figure}
Given a simple graph $G$, if $vw\in E(G)$, after complete expansion transfermation $f$, it follows from Lemma 2.1 that there is a unique edge connect complete graph $K_{d(v)}$ and $K_{d(w)}$.


\begin{lem}
If $H$ is the subgraph of $G$, then $\mathcal{E}_{c}(H)$ is the subgraph of $\mathcal{E}_{c}(G)$.
\end{lem}
\begin{proof}
Since $H$ is the subgraph of $G$, $V(H)\subseteq V(G)$ and $E(H)\subseteq E(G)$. $\forall v\in V(H)$, suppose $d_{H}(v)=m$, $d_{G}(v)=n$. It is clear that $m\leq n$. Suppose the edges incident with the vertex $v$ in $H$ are $e_{1},e_{2},\cdots,e_{m}$ and $e_{1},e_{2},\cdots,e_{m},e_{m+1},\cdots,e_{n}$  in $G$. For convenience, we denote $f_{H}$, $f_{G}$ as the complete expansion transformation of $H$ and $G$ respectively. By the definition of complete expansion transformation, $f_{H}(e_{i})=e_{i}\in \mathcal{E}_{c}(H)$, where $1\leq i\leq m$, $f_{G}(e_{j})=e_{j}\in \mathcal{E}_{c}(G)$, where $1\leq j\leq n$. $f_{H}(v)=K_{m}$, $f_{G}(v)=K_{n}$. $K_{m}$ is the subgraph of $K_{n}$. 
In consideration of the arbitrariness of $v$, the result follows from Lemma 2.1.

\end{proof}

The important tools for the proofs in the section 3 are the following classical theorem of Bernhart.

\begin{lem}$^{[2]}$
If $H$ is the subgraph of $G$, then $pn(H)\leq pn(G)$.
\end{lem}

\begin{lem}$^{[2]}$
Let $G$ be a connected graph. Then

(i) $pn(G)=0$ if and only if $G$ is a path.

(ii) $pn(G)\leq 1$ if and only if $G$ is outerplanar.

(iii) $pn(G)\leq 2$ if and only if $G$ is a subgraph of a hamiltonian planar graph.

\end{lem}

By Lemma 2.4, it is clear that non-planar graphs require at least three pages to embed. 

\section{Main Result}
~~~~~In this section, we study the book embedding of complete expansion graphs and get some results about it. It is obvious that the complete expansion graph of a path or a cycle is also a path or a cycle, hence the pagenumber of a path or a cycle do not change after complete expansion transformation. For general graph, the results are as follows.

\begin{thm}
If $H$ is the subgraph of $G$, then $pn(\mathcal{E}_{c}(H))\leq pn(\mathcal{E}_{c}(G))$.
\end{thm}
\begin{proof}
By the Lemma 2.2, if $H$ is the subgraph of $G$, then $\mathcal{E}_{c}(H)$ is the subgraph of $\mathcal{E}_{c}(G)$. By Lemma 2.3, $pn(\mathcal{E}_{c}(H))\leq pn(\mathcal{E}_{c}(G))$.
\end{proof}

\begin{thm}
For any simple graph $G$ with $n$ vertices, $\lceil\frac{\Delta(G)}{2}\rceil\leq pn(\mathcal{E}_{c}(G))\leq pn(G)+\lceil\frac{\Delta(G)}{2}\rceil$.
\end{thm}
\begin{proof}
(The lower bound)
 Let $m=\Delta(G)$, by the definition of complete expansion graph, $K_{m}$ is the subgraph of $\mathcal{E}_{c}(G)$. By Lemma 2.3, we get that $pn(\mathcal{E}_{c}(G))\geq pn(K_{m})=\lceil\frac{m}{2}\rceil=\lceil\frac{\Delta(G)}{2}\rceil$.

(The upper bound)
  Let $pn(G)=t$. Furthermore assume that $L: v_{1}, v_{2},\cdots, v_{n}$ is the optimal order of the vertices of $G$ and $P=\{P_{1},P_{2},\cdots,P_{t}\}$ is the $t$-page partition.
 If the edge $[v_{i},v_{i+1}]~(1\leq i\leq n)$ is not in $E(G)$ (up to mod $n$), we add it to any of the $t$ pages. Note that the cycle $v_{1}, v_{2},\cdots, v_{n},v_{1}$ for this $t$-book embedding of $G^{'}$ is a hamiltonian cycle, where $G^{'}$ results from $G$ by adding the missing edges.  There may be four cases for two edges of $G^{'}$ that belong to the same page:

 (i)the two edges are nested but have no common vertex;

 (ii)the two edges are nested and have one common vertex;

 (iii)the two edges are not nested and have no common vertex;

 (iv)the two edges are not nested but have a common vertex.

 By the definition of the complete expansion graph, we use the ``small" complete graphs $K_{d(v_{i})}$ to replace the original vertices $v_{i}$ on the spine. By the Lemma 2.1, there is only one edge incident with one vertex of $K_{d(v_{i})}$ and one vertex of another``small" complete graph. For case (ii) and (iv), the edges of $G^{'}$ that belong to the same page may cross after expansion transformation. We can exchange the order of the vertices of a ``small" complete graph on the spine by the symmetry of the  complete graph, so that the edges of $G^{'}$ in the same page are still in the same page after expansion transformation(see Figure 2(a-d)).

 \begin{figure}
\centering
\subfigure[]{
\label{Fig.sub.1}
\includegraphics[width=7cm]{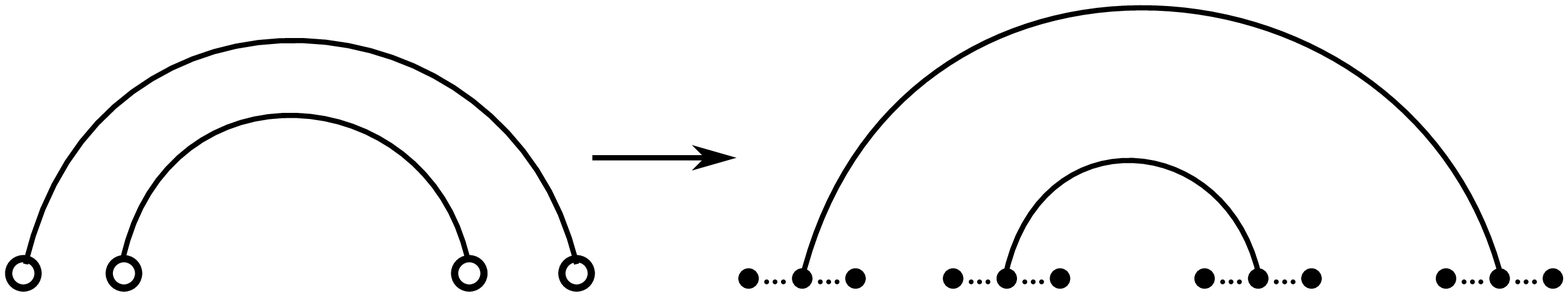}}
\subfigure[]{
\label{Fig.sub.1}
\includegraphics[width=7cm]{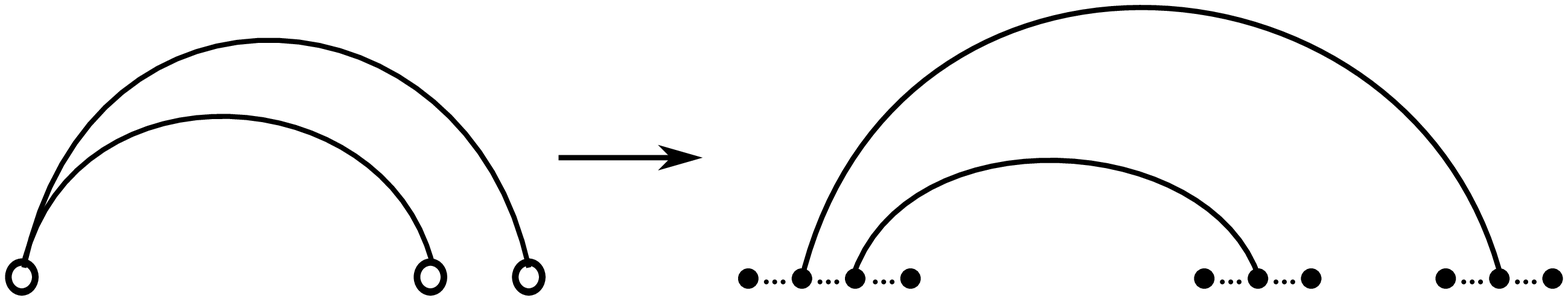}}

\subfigure[]{
\label{Fig.sub.1}
\includegraphics[width=7cm]{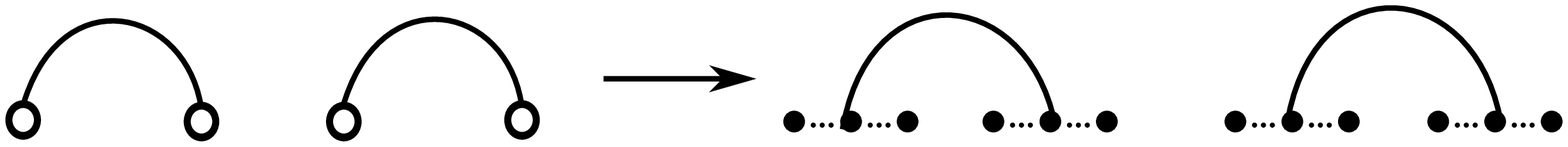}}
\subfigure[]{
\label{Fig.sub.1}
\includegraphics[width=6cm]{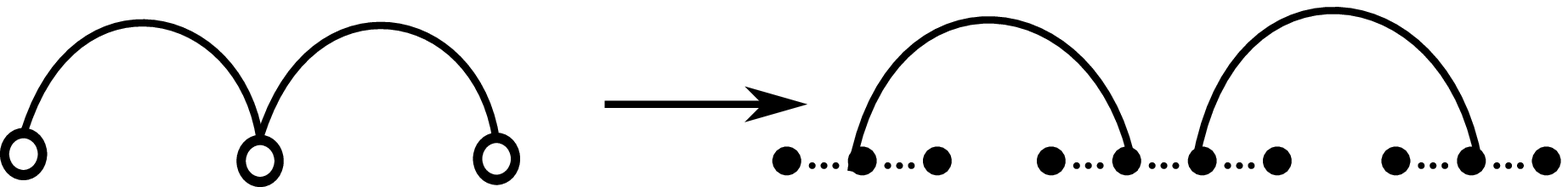}}
\subfigure[]{
\label{Fig.sub.1}
\includegraphics[width=11cm]{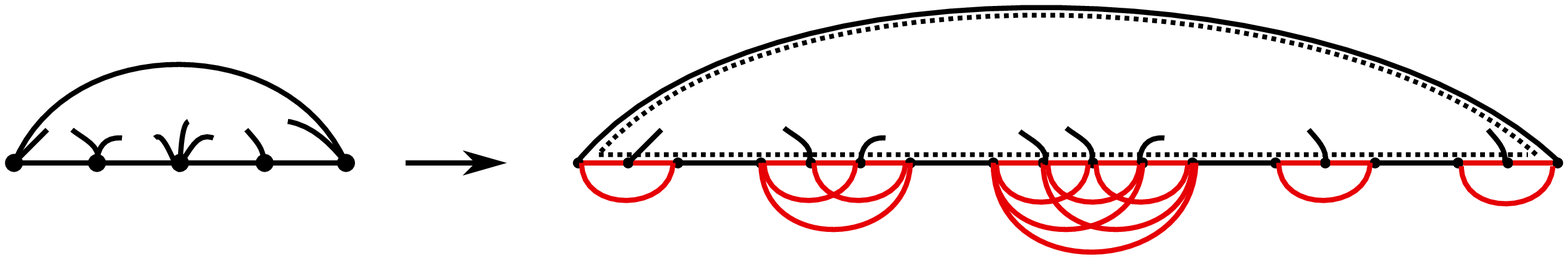}}
\caption{(a)The two edges in case (i). (b)The two edges in case (ii). (c)The two edges in case (iii). (d)The two edges in case (iv). (e)The complete expansion graph of $G^{'}$, the dotted curve shows the Jordan curve.
}.
\label{Fig.lable}
\end{figure}

 We get a Jordan curve $C$ in the plane(see Figure 2(e)). $C$ passes through each vertex of $\mathcal{E}_{c}(G)$
and partitions the plane into interior and exterior regions. The edges inside the curve can be divided into $t$ pages. The edges outside the curve which belong to these ``small" complete graphs $K_{d_{i}}$ can be divided into $\lceil\frac{m}{2}\rceil$ pages, where $m=\Delta(G)$ and $d_{i}\leq m$.

 In summary, $\lceil\frac{\Delta(G)}{2}\rceil\leq pn(\mathcal{E}_{c}(G))\leq pn(G)+\lceil\frac{\Delta(G)}{2}\rceil$.
\end{proof}

For a star graph $S_{m+1}$ with $m+1$ vertices,  whose maximum degree $\Delta(S_{m+1})=m$, we get that following result which shows that the lower bound of $pn(\mathcal{E}_{c}(G))$ given in Theorem 3.2 is tight and can not be improved.

\begin{thm}
For a star graph $S_{m+1}$  with $m+1$ vertices, $pn(\mathcal{E}_{c}(S_{m+1}))=\lceil\frac{m }{2}\rceil$.
\end{thm}
\begin{proof}
By Theorem 3.2, $pn(\mathcal{E}_{c}(S_{m+1}))\geq \lceil\frac{m}{2}\rceil$. For a star graph $S_{m+1}$, after complete expansion transformation, it is easy to see that the book embedding of $\mathcal{E}_{c}(S_{m+1})$ are mainly that of the complete graph $K_m$. The remained edges are all edges incident with the leaves of $\mathcal{E}_{c}(S_{m+1})$  which always can be embedded into appropriate pages (See an example for $m=8$ in Figure 3). So we have $pn(\mathcal{E}_{c}(S_{m+1}))\leq \lceil\frac{m}{2}\rceil$. Therefore, the result is established.

\begin{figure}
\centering
\subfigure[]{
\label{Fig.sub.1}
\includegraphics[width=4cm]{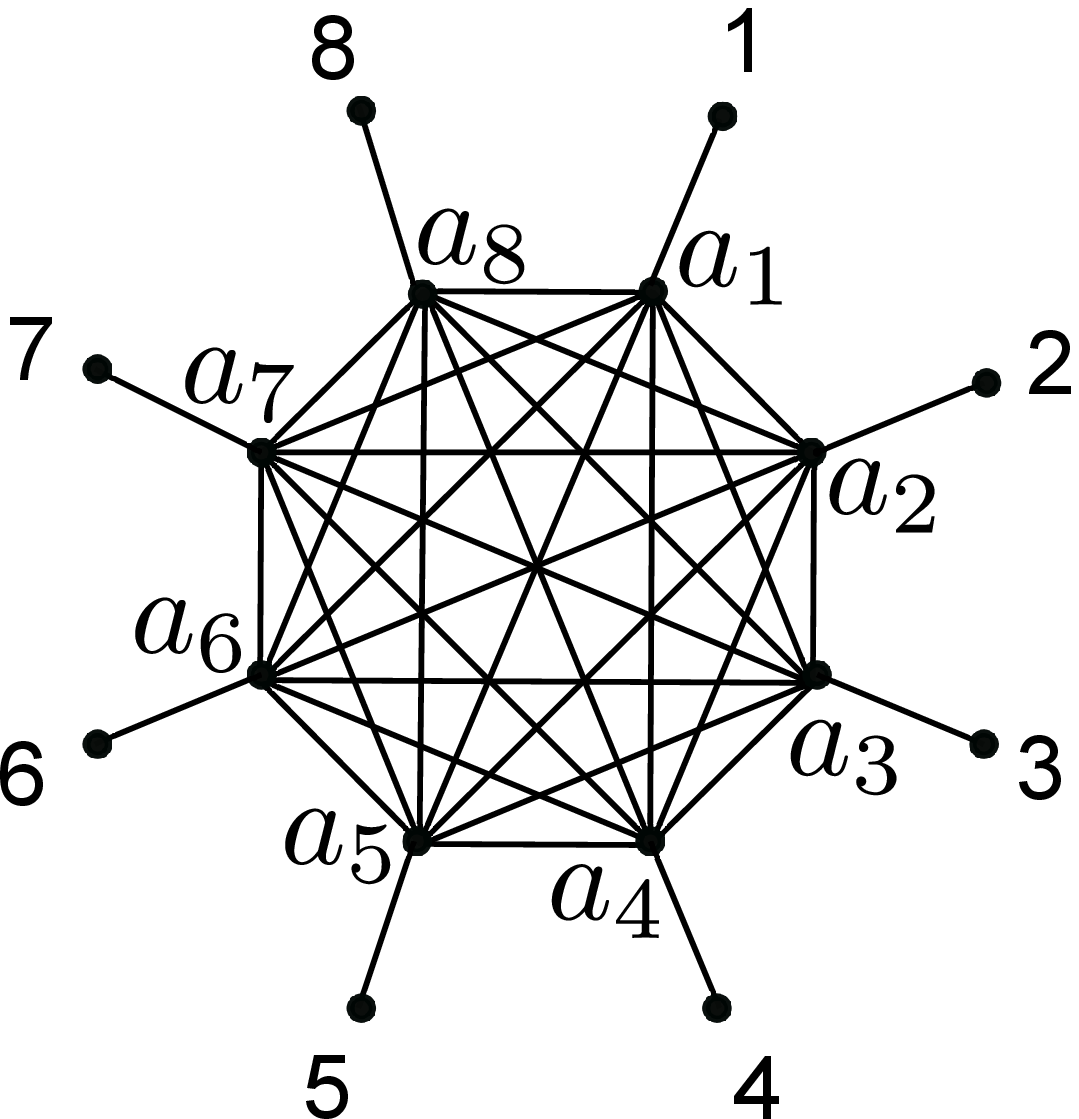}}
\subfigure[]{
\label{Fig.sub.1}
\includegraphics[width=7cm]{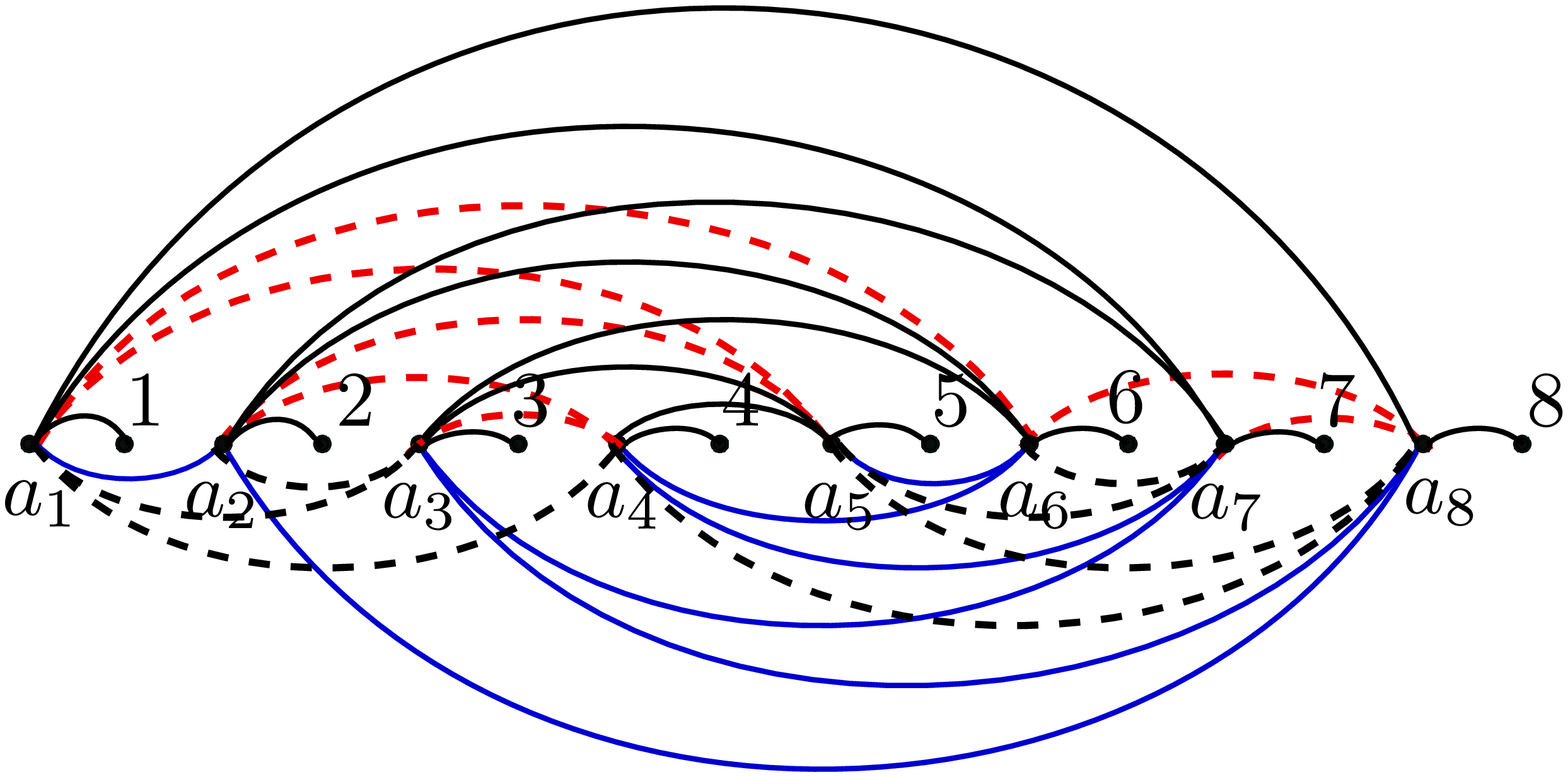}}
\caption{(a)$\mathcal{E}_{c}(S_{9})$. (b)4-page embedding of $\mathcal{E}_{c}(S_{9})$.
}.
\label{Fig.lable}
\end{figure}

\end{proof}


Before  investigating the pagenumber of the complete expansion graph of a tree,  let us recall the following two algorithms,  breadth-first search and  depth-first search$^{[16]}$,  which are involved in Theorem 3.4.

1. Breadth-first search (BFS): It starts at the tree root(or some arbitrary vertex of a graph, sometimes referred to as a 'search key') and explores all of the neighbor vertices at the present depth prior to moving on to the vertices at the next depth level. The label of a BFS-tree
is shown as Figure 4(a).

2. Depth-first search (DFS): 
One starts at the root (selecting some arbitrary vertex as the root in the case of a graph) and explores as far as possible along each branch
before backtracking. An example of the label of a BFS-tree is shown in Figure 4(b).
\begin{figure}
\centering
\subfigure[]{
\label{Fig.sub.1}
\includegraphics[width=4.5cm]{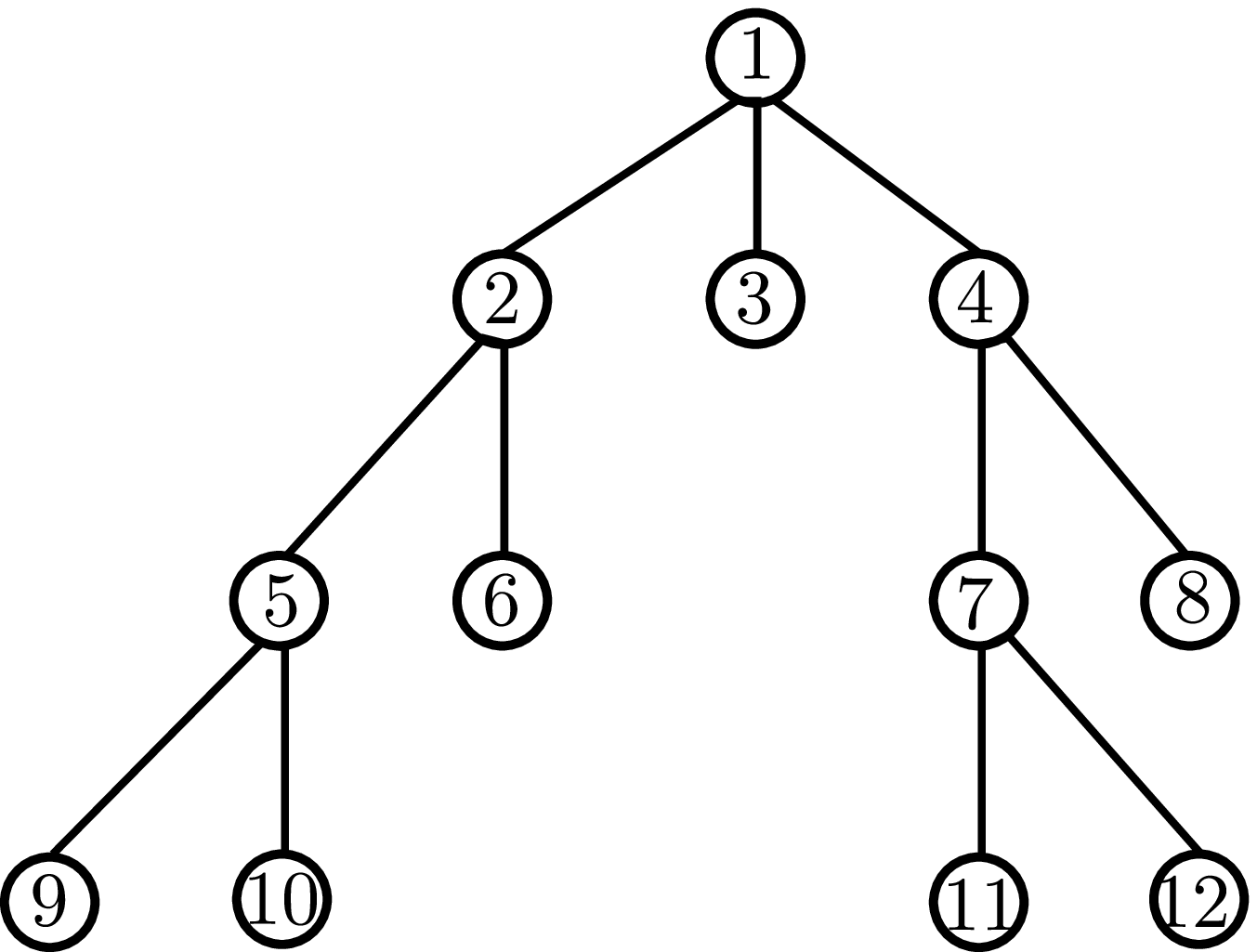}}
\subfigure[]{
\label{Fig.sub.1}
\includegraphics[width=4.5cm]{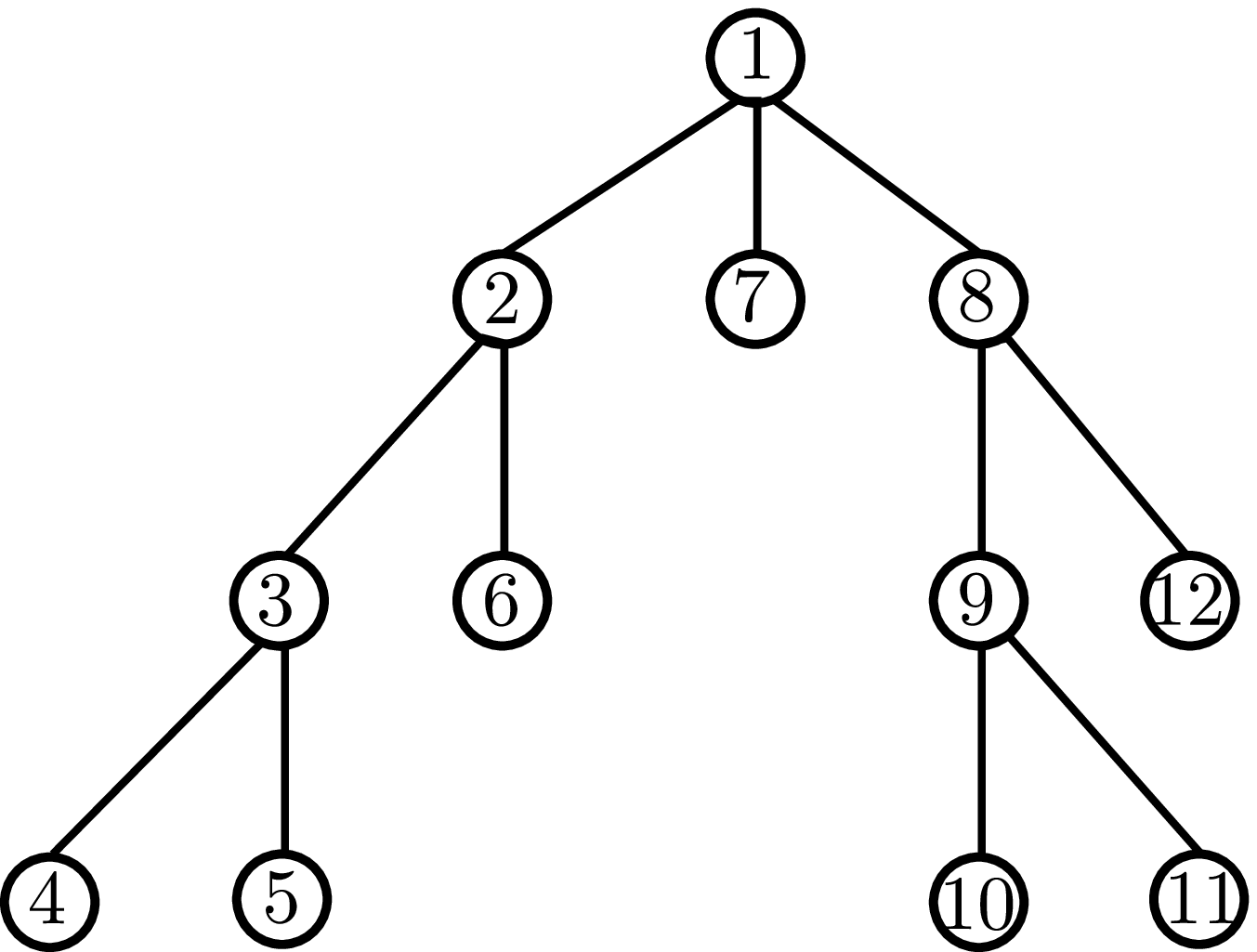}}
\caption{(a)The label of a breadth-first search tree. (b)The label of a depth-first search tree.
}.
\label{Fig.lable}
\end{figure}

\begin{thm}
For a tree $T$, $pn(\mathcal{E}_{c}(T))=\lceil\frac{\Delta(T)}{2}\rceil$.
\end{thm}
\begin{proof}(The lower bound)By Theorem 3.2, $pn(\mathcal{E}_{c}(T))\geq\lceil\frac{\Delta(T)}{2}\rceil$.

(The upper bound)Assume now that $T$ has $n$ vertices and is rooted at vertex $v_{0}$. For any vertex $v$ of $T$ except for $v_{0}$, denote by $p(v)$ the parent of $v$ in $T$. For a vertex $v$ of $T$ which is not a leaf, we define an order for its children: if $v^{'}$ and $v^{''}$ are children of $v$, then $v^{'}<v^{''}$ iff $v^{'}$ precedes $v^{''}$ in the counterclockwise order of the edges around $v$, when starting from $(v,p(v))$. 
We label the vertices of $T$ as they appear in the order of the DFS traversal where $1=v_{0}$ (see Figure 5(a)). After complete expansion transformation, the original vertices $i$ $(1\leq i\leq n)$ are replaced  by ``small" complete graphs with $d(i)$ vertices. It is easy to find that there is a vertex in the ``small" complete graph connects its parent complete graph in $\mathcal{E}_{c}(T)$ (see Figure 5(b)).
We start at this vertex and label the vertices  of the ``small" complete graph in the counterclockwise order by $ij$, where $1\leq j\leq d(i)$. We draw the ``small" complete graphs in the order implied by the Breadth First Search traversal of the complete expansion of tree. The first ``small" complete graph is drawn as in the case of complete graph. Each next ``small" complete graph is plugged into the drawing, as shown in Figure 5(c).
\begin{figure}
\centering
\subfigure[]{
\label{Fig.sub.1}
\includegraphics[width=5cm]{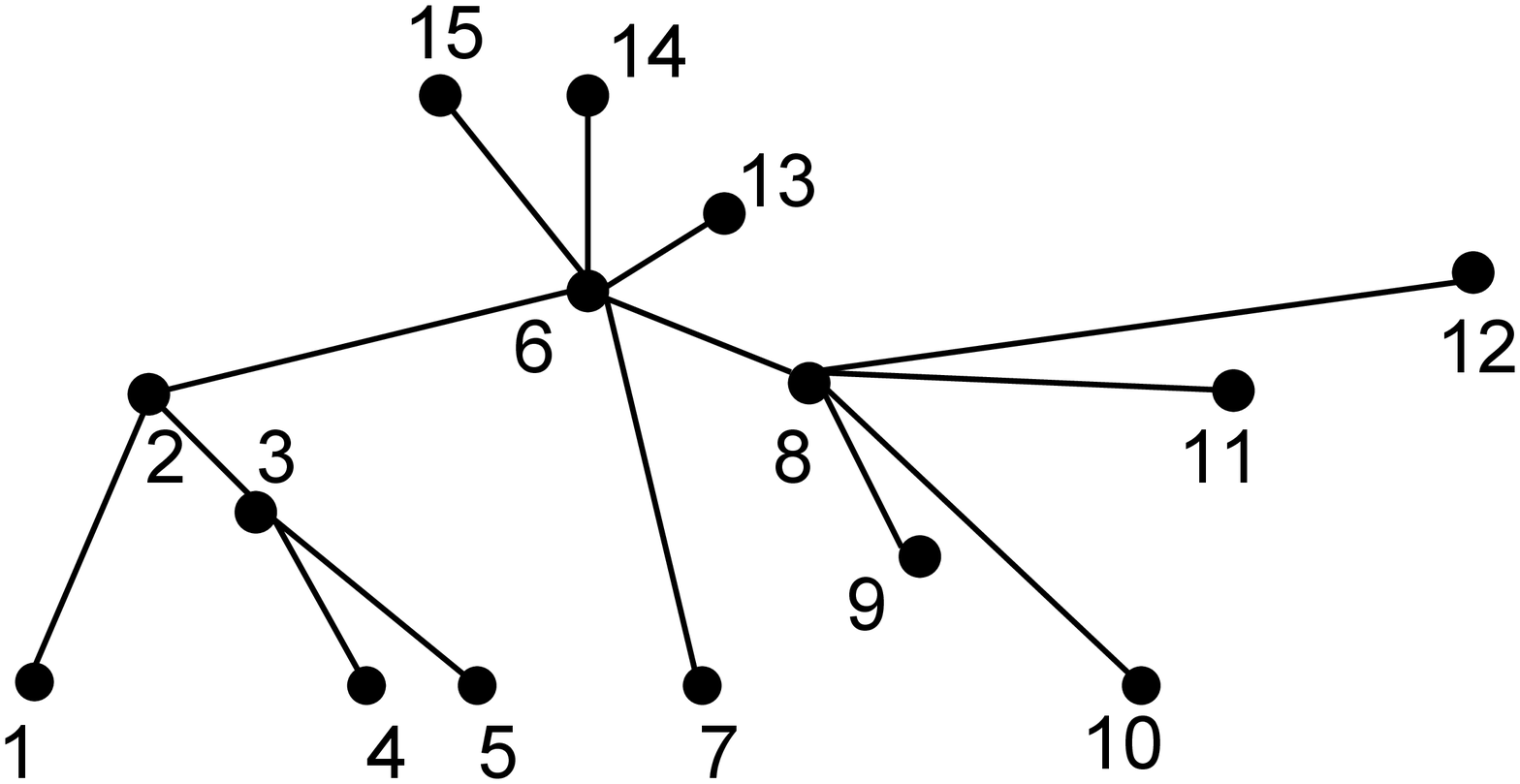}}
\subfigure[]{
\label{Fig.sub.1}
\includegraphics[width=7cm]{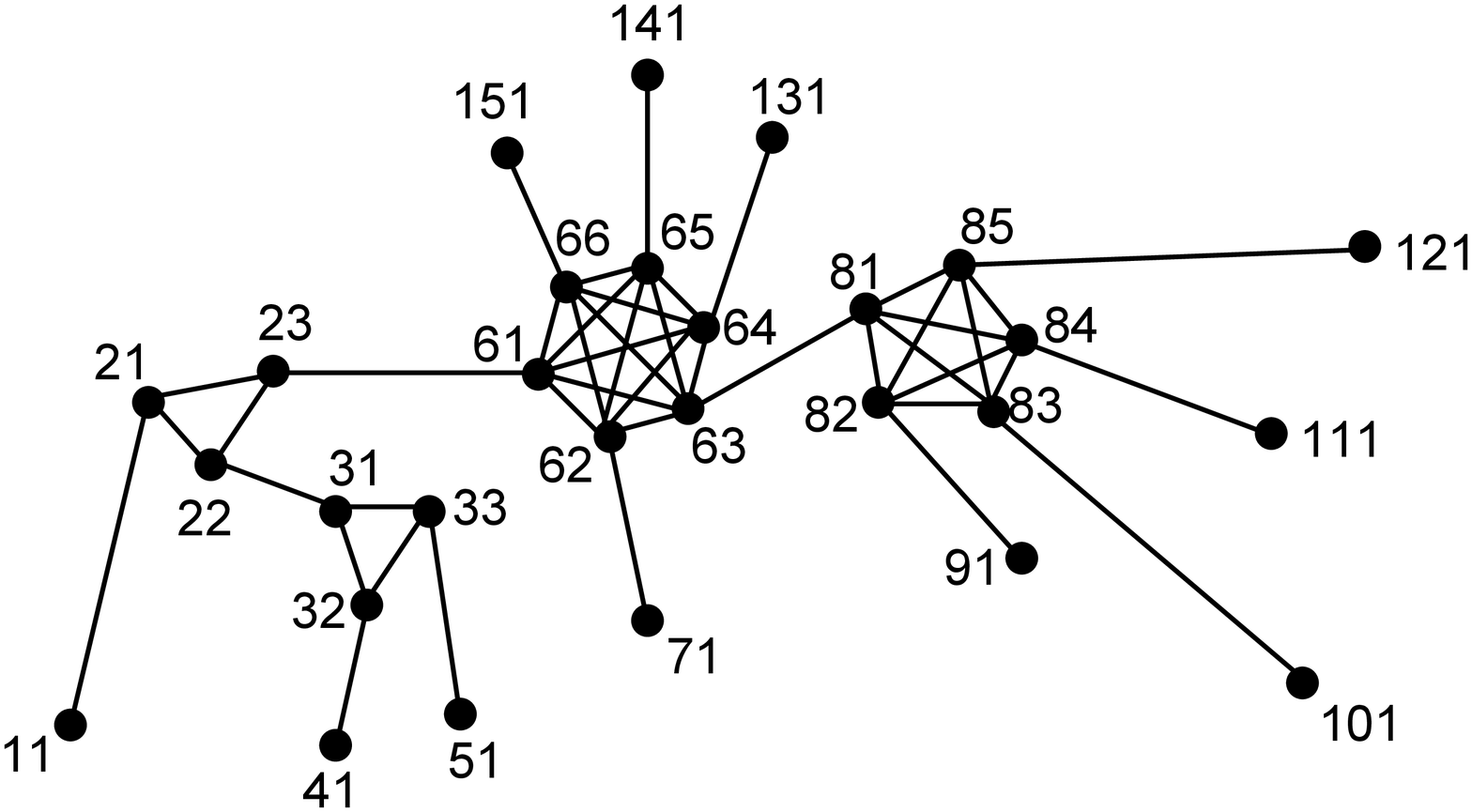}}
\subfigure[]{
\label{Fig.sub.1}
\includegraphics[width=12cm]{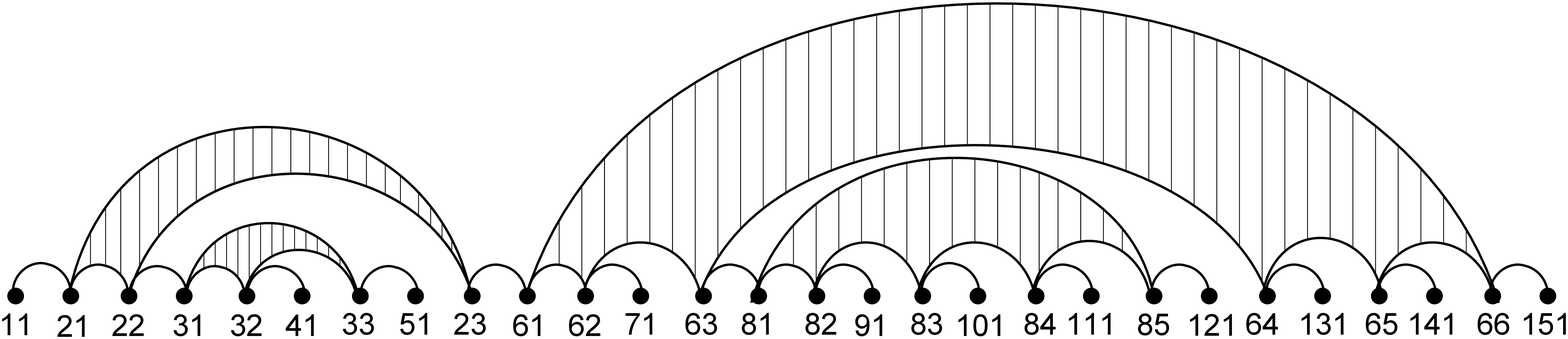}}
\caption{(a)A label of tree $T$. (b)A label of the complete expansion graph $\mathcal{E}_{c}(T)$ of tree $T$. (c)The position of the vertices of child complete graph and the parent complete graph. The shaded area indicates the edges of the ``small" complete graph.}
\label{Fig.lable}
\end{figure}

Every ``small" complete graph has a cycle $i1,i2,\cdots,id(i),i1$. The edges inside the cycle can be embedded in  $\lceil\frac{d(i)}{2}\rceil$ pages, and the edges on the cycle can be embedded in some appropriate pages of these $\lceil\frac{d(i)}{2}\rceil$ pages. The edges belonging to different complete graphs are nested with each other or side by side, so we only need to consider the pagenumber of the largest complete graph with $\Delta(T)$ vertices. Then $pn(\mathcal{E}_{c}(T))\leq\lceil\frac{\Delta(T)}{2}\rceil$.

Thus we arrive at the conclusion that $pn(\mathcal{E}_{c}(T))=\lceil\frac{\Delta(T)}{2}\rceil$.
\end{proof}

In [4], the pagenumber of M$\ddot{o}$bius Ladder $M_{h}$ and the Petersen graph $P$ has been showed that $pn(M_{h})=3$ and $pn(P)=3$. With the similar method, we get the pagenumber of the complete expansion graph of $M_{h}$ and $P$ as following:
 \begin{figure}
\centering
\subfigure[]{
\label{Fig.sub.1}
\includegraphics[width=6cm]{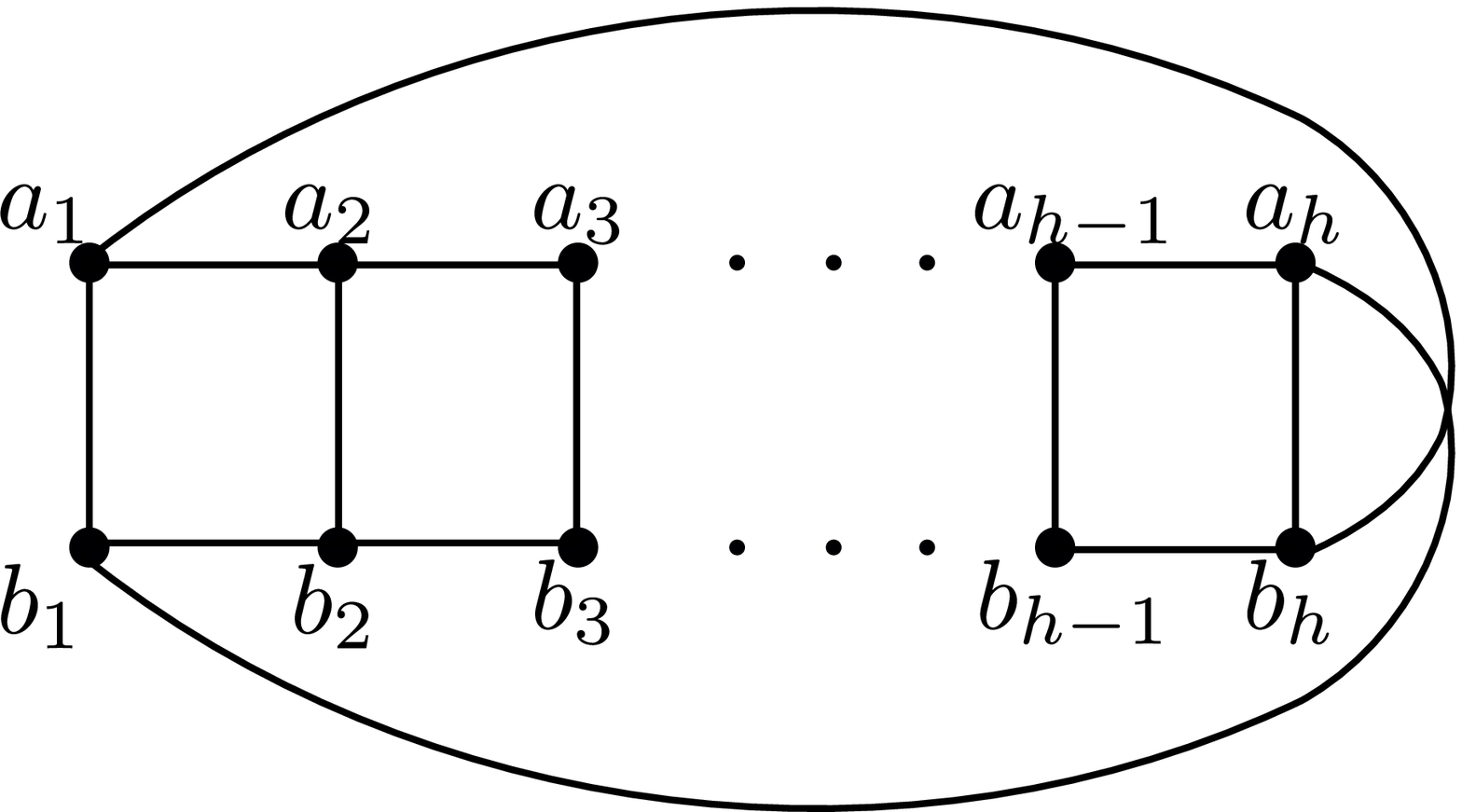}}
\subfigure[]{
\label{Fig.sub.1}
\includegraphics[width=7cm]{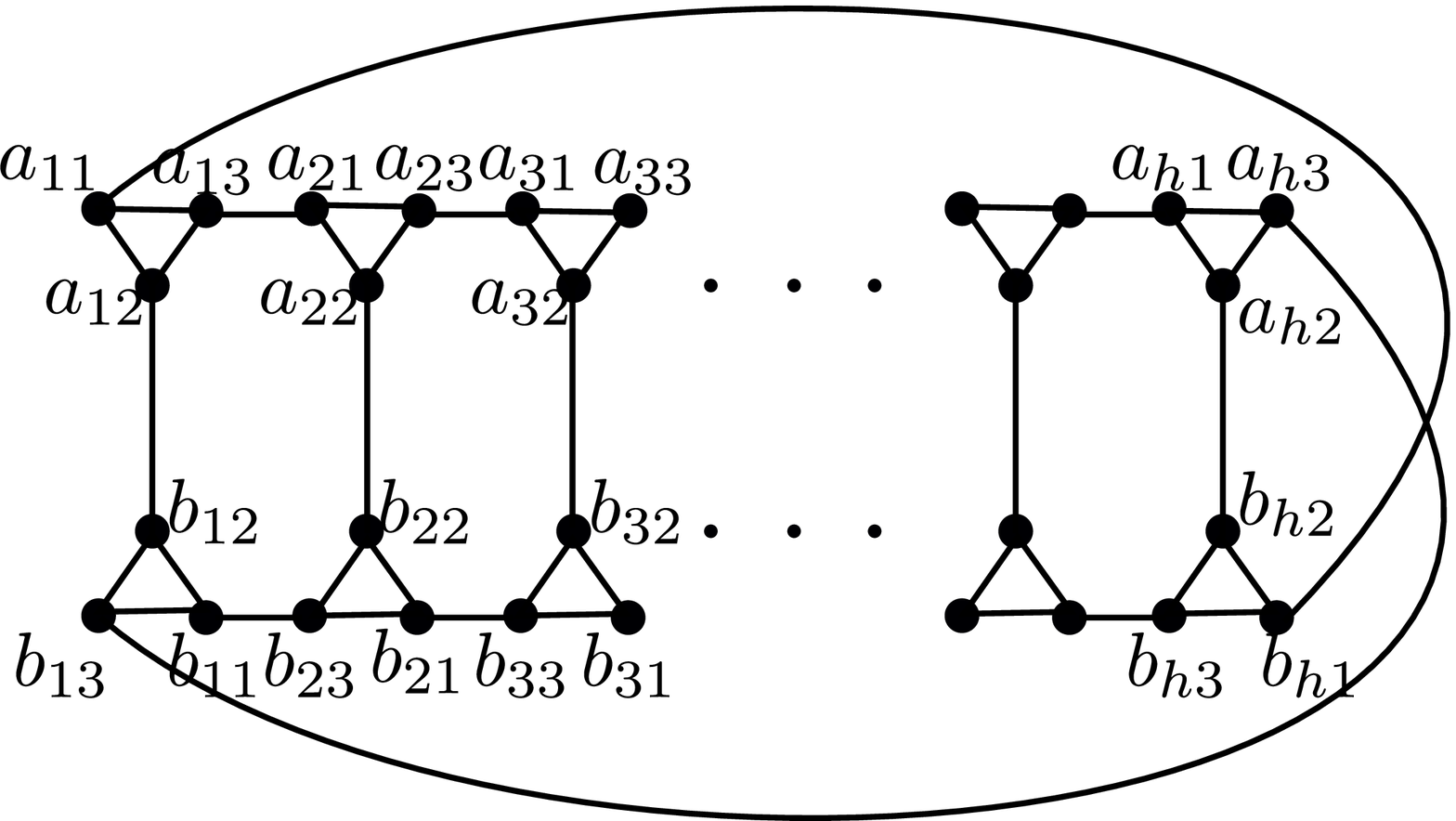}}
\subfigure[]{
\label{Fig.sub.1}
\includegraphics[width=14cm]{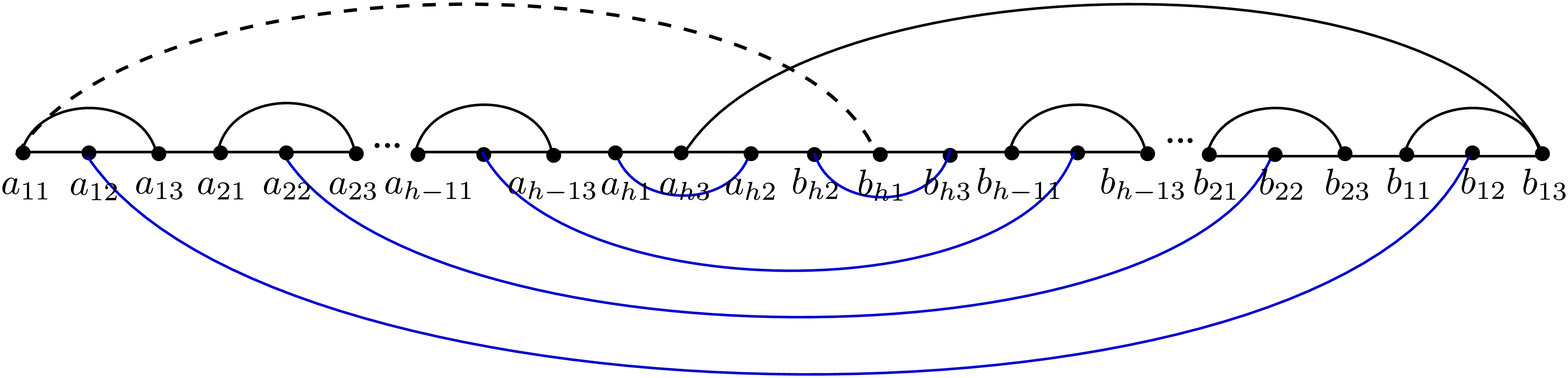}}
\caption{(a)$M\ddot{o}bius$ Ladder $M_{h}$. (b)The complete expansion graph of $M_{h}$. (c)The book embedding of $\mathcal{E}_{c}(M_{h})$.
}.
\label{Fig.lable}
\end{figure}

\begin{thm}
For $M\ddot{o}bius$ Ladder $M_{h}$,  $pn(\mathcal{E}_{c}(M_{h}))=3$.
\end{thm}
\begin{proof}
(The lower bound) When $h=3$, $M_{h}=K_{3,3}$. For $h\geq 3$, $K_{3,3}$ is the  minor of $M_{h}$. $M_{h}$ is the minor of $\mathcal{E}_{c}(M_{h})$. Thus, $\mathcal{E}_{c}(M_{h})$ contains $K_{3,3}$ as a minor and $\mathcal{E}_{c}(M_{h})$ is not a planar graph. By Lemma 2.4, $pn(\mathcal{E}_{c}(M_{h}))\geq3$.

(The upper bound)We can obtain a printing cycle by placing vertices as Figure 6 and construct a three-page partition for edges as:

$P_{1}=\{a_{11}b_{h1}\}$

$P_{2}=\{b_{13}a_{h3}, a_{i1}a_{i3},  b_{i1}b_{i3}| i=1,2,\cdots,h-1\}$

$P_{3}=\{The~other~edges\}$

The proof of the theorem is now completed.

 \begin{figure}
\centering
\subfigure[]{
\label{Fig.sub.1}
\includegraphics[width=4.5cm]{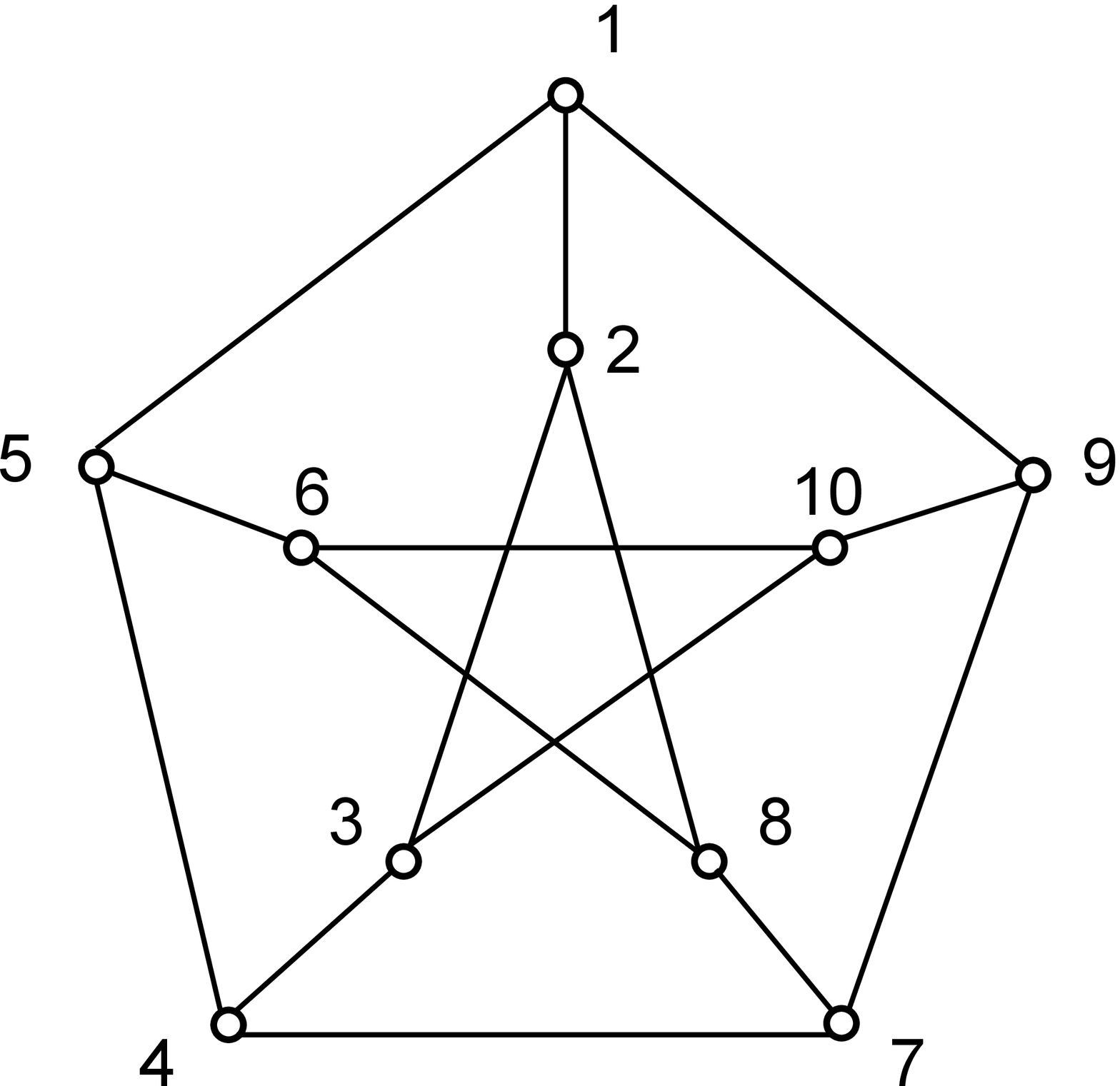}}
\subfigure[]{
\label{Fig.sub.1}
\includegraphics[width=4.5cm]{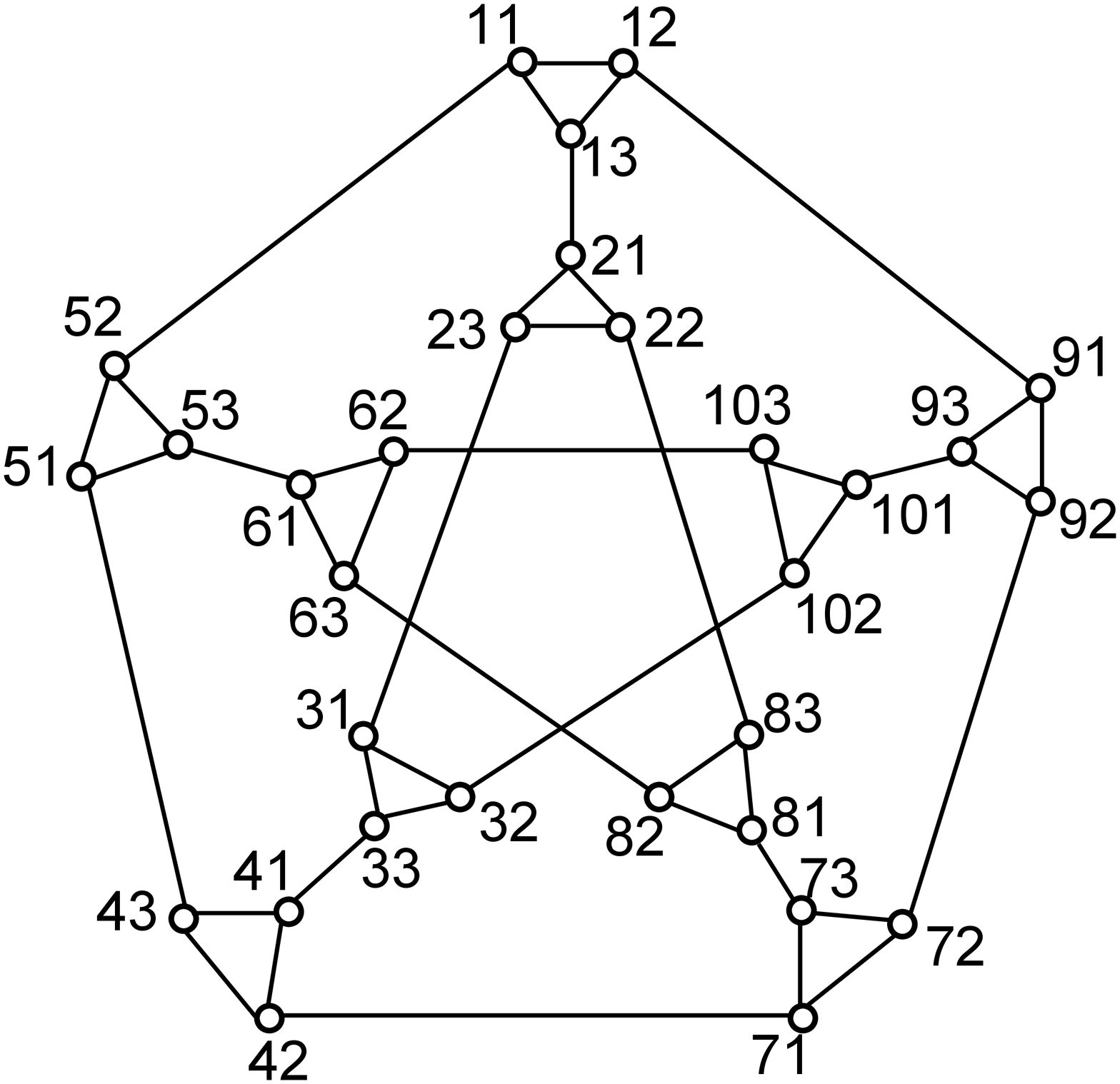}}
\subfigure[]{
\label{Fig.sub.1}
\includegraphics[width=13cm]{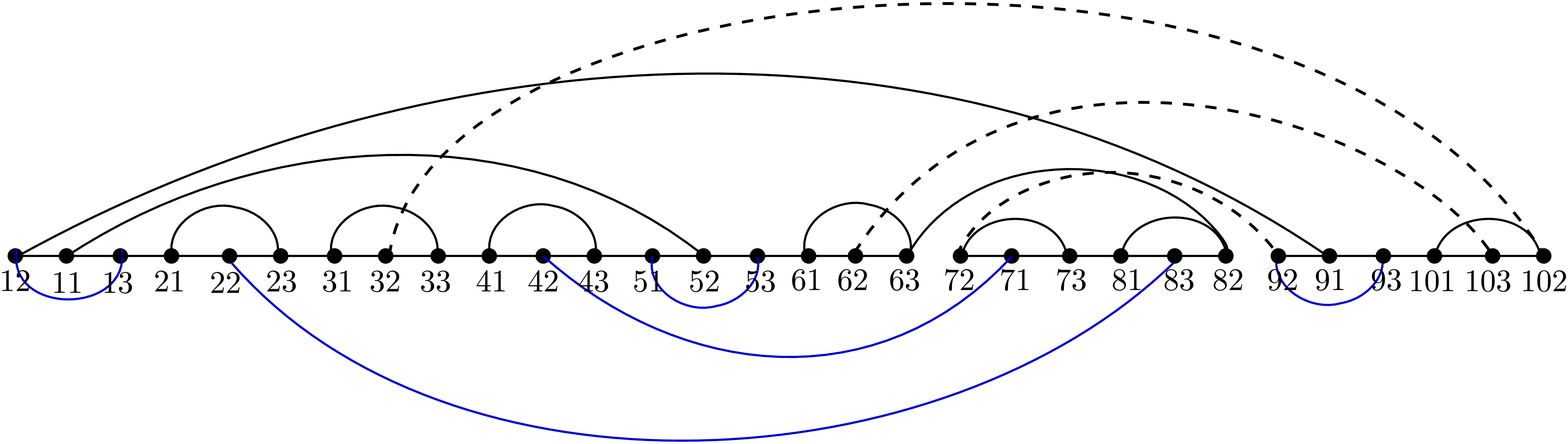}}
\caption{(a)Petersen graph $P$. (b)The complete expansion graph of $P$. (c)The book embedding of $\mathcal{E}_{c}(P)$.
}.
\label{Fig.lable}
\end{figure}

\end{proof}

\begin{thm}
For the Petersen graph $P$, $pn(\mathcal{E}_{c}(P))=3$.
\end{thm}
\begin{proof}
 The expansion graph of $P$ contains $K_{5}$ as a minor (See Figure 7(b)) and hence $\mathcal{E}_{c}(P)$ is not a planar graph. By Lemma 2.4, $pn(\mathcal{E}_{c}(P))\geq3$. On the other hand, Figure 7(c) shows an three-page book embedding of the complete expansion graph $\mathcal{E}_{c}(P)$. Thus the result is established.

%

\end{proof}


The complete graph $K_{n}$ can be embedded in $\lceil\frac{n}{2}\rceil$ pages. We consider that the pagenumber of complete expansion graph of $K_{n}$ and get the  following result:

\begin{thm}
For $n=2m,m\in Z^{+}$, $pn(\mathcal{E}_{c}(K_{n}))=m$; For $n=2m+1,m\in Z^{+}$, $m\leq pn(\mathcal{E}_{c}(K_{n}))\leq m+1$.
\end{thm}
\begin{proof}

The complete expansion transformation of $K_{n}$ is to replace the original $n$ vertices in $K_{n}$ by ``small" complete graphs $K_{n-1}$, $|V(\mathcal{E}_{c}(G))|=n(n-1)$.  Its edge set is $E_{1}\cup E_{2}$ , where $E_{1}=\{(v,e)(v,f)|v\in V(K_{n}),e,f\in E(K_{n})\}, E_{2}=\{(v,e)(w,e)|v,w\in V(K_{n}), e\in E(K_{n})\}$. $E_{2}$ corresponds to the edges of the original complete graph $K_{n}$. $E_{1}$ corresponds to the edges of $n$ newly generated complete graph $K_{n-1}$.  $|E(\mathcal{E}_{c}(G))|=|E(K_{n})|+n|E(K_{n-1})|=\frac{n(n-1)^{2}}{2}$.
The vertices in the complete expansion graph are represented by $ij$, where $1\leq i\leq n, 1\leq j\leq n-1$.

Case 1: $n=2m$.

(The lower bound)
By Theorem 3.2, it is easy to see that $pn(\mathcal{E}_{c}(K_{n}))\geq \lceil\frac{n-1}{2}\rceil=\lceil\frac{2m-1}{2}\rceil=m$.

(The upper bound)
When $n=2m$, we construct a $m$-page partition for edges of $E_{2}$ as:

$P_{1}=\{(11,2m(2m-1)),(2m(2m-2),22),(23,(2m-1)(2m-3)),\cdots,(m(2m-1),(m+1)1)\}$;

$P_{2}=\{(33,2m(2m-3)),(2m(2m-4),44),(45,(2m-1)(2m-5)),\cdots,((m+1)(2m-1),(m+2)1)$ and $(32,1(2m-2)),(1(2m-1),21)\}$;

$P_{3}=\{(55,2m(2m-5)),(2m(2m-6),66),(67,(2m-1)(2m-7)),\cdots,((m+2)(2m-1),(m+3)1)$ and $(54,1(2m-4)),(1(2m-3),43),(42,2(2m-2)),(2(2m-1),31)\}$;

$\cdots$

$P_{m-1}=\{((2m-3)(2m-3),2m3),(2m2,(2m-2)(2m-2)),((2m-2)(2m-1),(2m-1)1)$ and $((2m-3)(2m-4),14),(15,(2m-4)(2m-5)),((2m-4)(2m-6),26),\cdots,((m-2)(2m-1),(m-1)1)\}$;

$P_{m}=\{((2m-1)(2m-1),2m1)$ and $((2m-1)(2m-2),12),(13,(2m-2)(2m-3)),((2m-2)(2m-4),24),\cdots,((m-1)(2m-1),m1)\}$.

We can find that the following edge sets can be embedded in the above $m$ pages respectively.

$\{(i1,i(2m-1)),(i(2m-1),i2),(i2,i(2m-2)),\cdots,(i(m-1),i(m+1)),(i(m+1),im)\}$;

$\{(i3,i(2m-1)),(i(2m-1),i4),(i4,i(2m-2)),\cdots,(im,i(m+2)),(i(m+2),i(m+1))\}$ and $\{(i1,i2)\}$;

$\cdots$

$\{(i(2m-3),i(2m-1)),(i(2m-1),i(2m-2))\}$ and $\{(i(2m-4),i1),(i1,i(2m-5)),(i(2m-5),i2),\cdots,(i(m-1),i(m-2))\}$;

$\{(i(2m-2),i1),(i1,(i(2m-3)),\cdots,(i(m-2),im),(im,i(m-1))\}$;

The edges in each page are not intersected(see Figure 8 for case $m=3$).
\begin{figure}[htbp]
\centering
\includegraphics[height=5cm, width=0.8\textwidth]{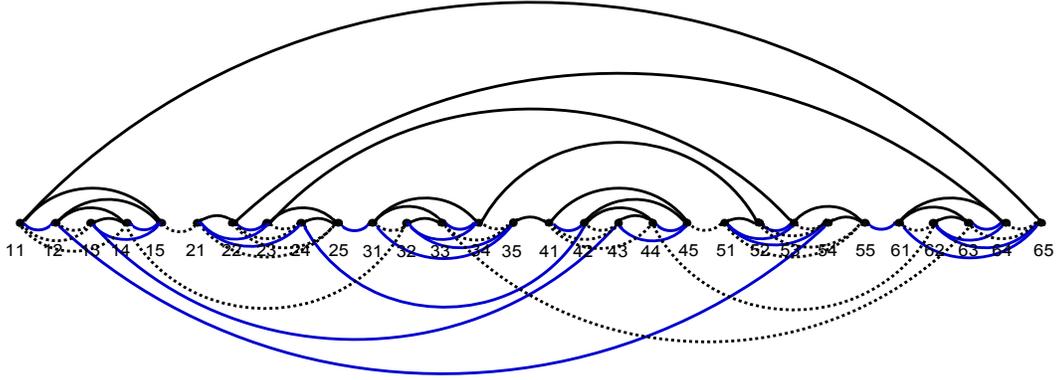}
\caption{The book embedding of $\mathcal{E}_{c}(K_{6})$}
\label{1}
\end{figure}

Case 2: $n=2m+1$.

(The lower bound)
By Theorem 3.2, $pn(\mathcal{E}_{c}(K_{n}))\geq \frac{n-1}{2}=m$.

(The upper bound) We construct a $m+1$-page partition for edges of $E_{2}$ as:

$P_{1}=\{(11,(2m+1)(2m)),((2m+1)(2m-1),22),\cdots,(m(2m-1),(m+2)2),((m+2)1,(m+1)(2m))\}$;

$P_{2}=\{(33,(2m+1)(2m-2)),((2m+1)(2m-3),44),\cdots,((m+1)(2m-1),(m+3)2),((m+3)1,(m+2)(2m))$ and $(21,1(2m))\}$;

$P_{3}=\{(55,(2m+1)(2m-4)),((2m+1)(2m-5),66),(67,(2m)(2m-6)),\cdots,((m+2)(2m-1),(m+4)2),((m+4)1,(m+3)(2m))$ and $(43,1(2m-2)),(1(2m-1),32),(31,2(2m))\}$;

$\cdots$

$P_{m}=\{((2m-1)(2m-1),(2m+1)2),((2m+1)1,(2m)(2m))$ and $((2m-2)(2m-3),14),(15,(2m-3)(2m-4)),((2m-3)(2m-5),26),\cdots,((m-2)(2m-1),m2),(m1,(m-1)(2m))\}$;

$P_{m+1}=\{((2m)(2m-1),12),(13,(2m-1)(2m-2)),((2m-1)(2m-3),24),\cdots,((m-1)(2m-1),(m+1)2),((m+1)1,m(2m))\}$.

The edge sets of $E_{1}$ include:

$\{(i(2m),i1),(i1,i(2m-1)),(i(2m-1),i2),\cdots,(i(m-1),i(m+1)),(i(m+1),im)\}$;

$\{(i(2m-2),i1),(i1,i(2m-3)),\cdots,(i(m-2),im),(im,i(m-1))\}$;

$\{(i(2m-4),i1),(i1,i(2m-5)),\cdots,(i(m-1),i(m-2))\}$ and $\{(i(2m-2),i(2m)),(i(2m),i(2m-1))\}$;

$\cdots$

$\{(i2,i1)\}$ and $\{(i4,i(2m)),(i(2m),i5),\cdots,(i(m+3),i(m+2))\}$;

$\{(i2,i(2m)),(i(2m),i3),\cdots,(i(m+2),i(m+1)\}$;

\begin{figure}[htbp]
\centering
\includegraphics[height=4.5cm, width=0.7\textwidth]{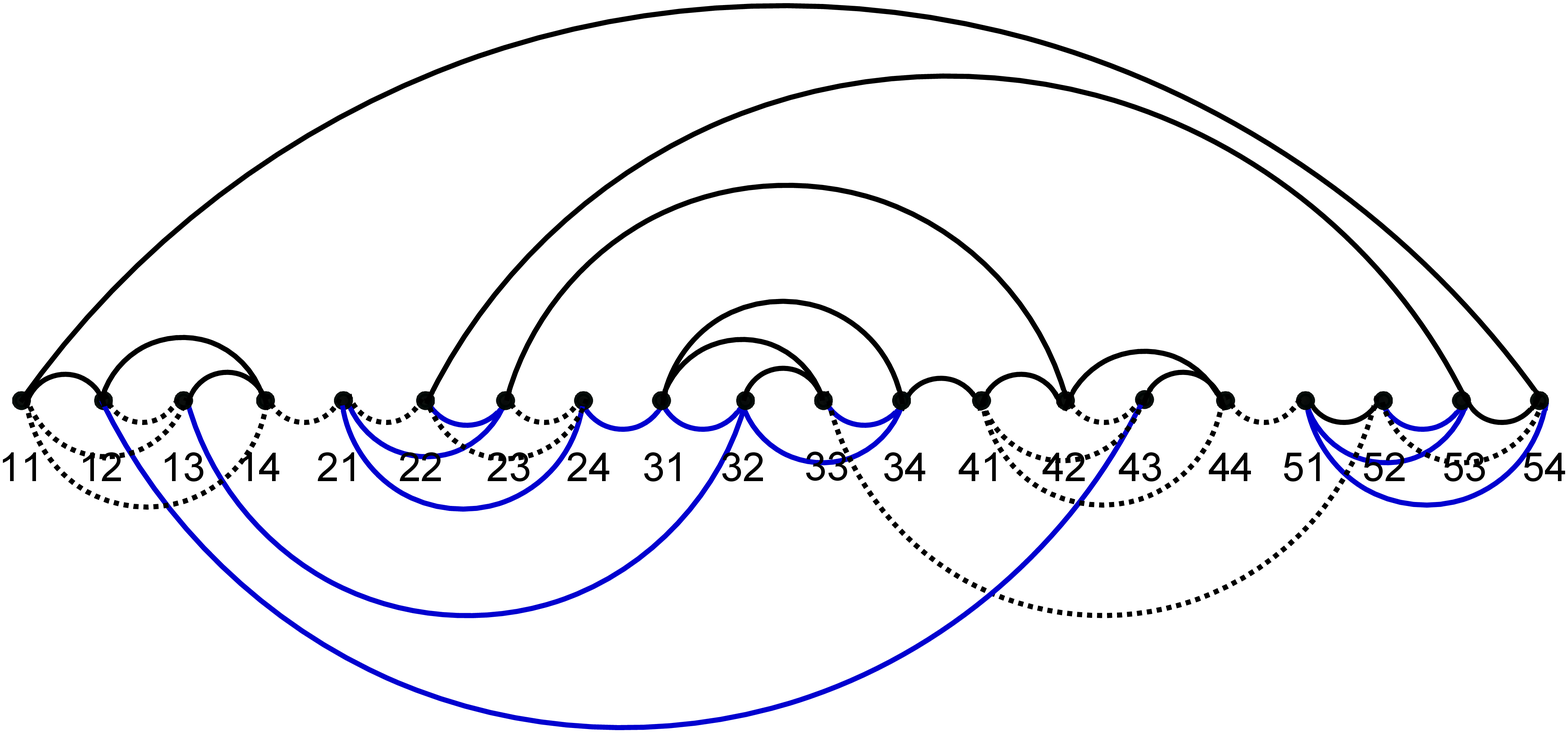}
\caption{The book embedding of $\mathcal{E}_{c}(K_{5})$}
\label{1}
\end{figure}

These edge sets can be embedded in the page partition of $E_{2}$. The edges of each page are not intersected(see Figure 9 for the case $m=2$).

The proof of the theorem is now completed.
\end{proof}

\section{Open problem}
~~~~~In section 3, we have obtained that the difference between the upper and lower bounds of the pagenumber of $\mathcal{E}_{c}(K_{2m+1})$ given by Theorem 3.7 is one.  When $m=2$, by the Theorem 3.7 and the non-planarity of $\mathcal{E}_{c}(K_{5})$, it is easy to get that $pn(\mathcal{E}_{c}(K_{5}))=3$. The natural open problem is whether the lower bound of the pagenumber of $\mathcal{E}_{c}(K_{2m+1})$ can be improved for any $m$.
We conjecture that this inequality in Theorem 3.7 can become equality. In the other words, we propose the following conjecture:

\begin{conj}
 For $n=2m+1,m\in Z^{+}$, $pn(\mathcal{E}_{c}(K_{n}))=\lceil\frac{n}{2}\rceil$.
\end{conj}




\begin{thebibliography}{99}

\bibitem{BS}
P. C. Kainen. Some recent results in topological graph theory[J]. Graphs and combinatarics, 1974, 406: 76-108.

\bibitem{BK}
F. Bernhart, P. C. Kainen. The book thickness of a graph[J]. Journal of Combinatorial Theory, Series B, 1979, 27: 320-331.



\bibitem{ZX}
B. Zhao, Y. Z. Tian, J. X. Meng. Embedding semistrong product of paths and cycles in books[J]. Natural Science of Hunan Normal University, 2015, 38(06): 73-77.


\bibitem{CL}
X. L. Li. Book embedding of graphs[D]. Zhengzhou: Zhengzhou university, 2002.

\bibitem{DW}
F. R. K. Chung, F. T. Leighton, A. Rosenberg. Embedding graphs in books: A layout problem with applications to VLSI design[J]. SIAM Journal on Algebraic and Discete Methods, 1987, 8(1): 33-58.

\bibitem{ET}
E. De Klerk, D. V. Pasechnik, G. Salazar. Book drawings of complete bipartite graphs[J]. Discrete Applied Mathematics, 2014, 167(4): 80-93.

\bibitem{EN}


H. Enomoto, T. Nakamigawa, K. Ota. On the pagenumber of complete bipartite graphs[J]. Joural of Combinatorial Theory Series B, 1997, 71: 111-120.


\bibitem{GH}
J. Balogh, G. Salazar. Book embedding of regular graph[J]. SIAM Journal on Discrete Mathematics, 2016, 29(2): 811-822.


\bibitem{JF}
T. Endo. The pagenumber of toroidal graphs is at most seven[J]. Discrete Mathematics, 1997, 175: 87-96.

\bibitem{HI}
J. F. Fang, K. C. Lai. Embedding the incomplete hypercube in books[J]. Information Processing Letters, 2005, 96(1): 1-6.

\bibitem{KC}
B. Zhao, Y. Z. Tian, J. X. Meng. Embedding semistrong product of paths and cycles in books[J]. Hunan: Journal of Natural Science of Hunan Normal University, 2015, 38(06): 73-77.


\bibitem{MS}
J. Yang, Z. L. Shao, Z. G. Li. Embedding cartesian product of some graphs in books[J]. Communications in Mathematical Research, 2018, 34(3): 253-260.


\bibitem{MS}

A. Yongga, S. qin. The expansion graph and the properties of its spectrum[J]. Journal of Baoji Unvercity of Arts And Science (Natural Science), 2009, 29(1): 1-3.

\bibitem{MS}
W. Jiang, A. Yongga. Expansion Graph and Their Chromatic Number numbers[J]. Journal of Inner Mongolia Normal University (Natural Science Edition), 2011, 40(3): 22-234.

\bibitem{MS}
L. Ying, A. Yongga. Hamiltonicity of the complete expansion graph[J]. Journal of Baoji Univercity of Arts and Science (Natural Science), 2011, 31(4): 24-28.

\bibitem{MS}
Bondy J, Murty U. Graph theory [M]. London: Springer, 2008.
\end{thebibliography}
\end{document}